\documentclass[10pt,a4paper, leqno]{article}
\usepackage[utf8]{inputenc}
\usepackage{amsmath}
\usepackage{amsfonts}
\usepackage{amssymb}
\usepackage{comment}
\usepackage{hyperref}
\usepackage{enumitem,pstricks,xy,pst-node}
\xyoption{all}
\usepackage{amsthm}
\usepackage{booktabs} 
\usepackage[color]{changebar} 
\cbcolor{red} 
\usepackage{mathrsfs} 
\usepackage{verbatim} 
\usepackage[disable] 
{todonotes}
\usepackage[normalem]{ulem} 
\usepackage{caption} 
\usepackage{nicefrac} 
\usepackage{enumitem} 

\newtheorem{theorem}{Theorem}[section]
\newtheorem{lemma}[theorem]{Lemma}
\newtheorem{proposition}[theorem]{Proposition}
\newtheorem{corollary}[theorem]{Corollary}
\newtheorem{definition}[theorem]{Definition}
\theoremstyle{definition}

\newtheorem{remark}[theorem]{Remark}
\newtheorem*{ack}{Acknowledgements}

\theoremstyle{remark}
\newtheorem{example}[theorem]{Example}

\newcommand{\PP}{\mathbb{P}}
\newcommand{\Mac}{{\texttt {Macaulay2}}}

\def\cU{{\mathcal U}}
\def\zero{\mathscr{Z}}

\def\PP{\mathbf P}

\def\CC{\mathbb{C}}

\def\cE{{\mathcal E}}

\def\cW{{\mathcal W}}
\def\cO{{\mathcal O}}

\def\cQ{{\mathcal Q}}

\def\af1{\mathbf{aff}_1}

\DeclareMathOperator{\im}{Im}

\DeclareMathOperator{\Hom}{Hom}

\DeclareMathOperator{\Sing}{Sing}
\DeclareMathOperator{\codim}{codim}

\DeclareMathOperator{\Ker}{Ker}
\DeclareMathOperator{\Gr}{Gr}
\DeclareMathOperator{\IGr}{IGr}
\DeclareMathOperator{\OGr}{OGr}
\DeclareMathOperator{\GL}{GL}

\newcommand{\ladi}{\begin{lastadd}}
\newcommand{\ladf}{\end{lastadd}}
\newcommand{\lrei}{\begin{lastrem}}
\newcommand{\lref}{\end{lastrem}}
\newenvironment{lastadd}
{\cbstart\color{red}}
{\todo{red to remove}\cbend}
\newenvironment{lastrem}
{\cbstart\color{yellow}}
{\cbend}

\author{Vladimiro Benedetti\thanks{Aix-Marseille Universit\'e, CNRS, Centrale Marseille, I2M, UMR 7373, 13453 Marseille, France.}}
\title{Crepant resolutions of orbit closures in Quiver representations of type $A_n$ and $D_4$}
\begin{document}
\maketitle

\begin{abstract}
We study resolutions of singularities of orbit closures in quiver representations. We consider certain resolutions of singularities which have already been constructed by Reineke, and we determine under which conditions they are crepant. More specifically, we deal with quivers of type $A_3$, $A_n$, $D_4$. Then, as an application of the results we have found, we construct several Fano and Calabi-Yau varieties as orbital degeneracy loci.
\end{abstract}
\setcounter{tocdepth}{1}

\section{Introduction}

The study of orbits inside representations of algebraic groups is a quite lively subject that has been confronted by many authors; for instance in \cite{Weyman} the so-called geometric technique is developed, and many applications are given, for example to the study of determinantal varieties. Using resolutions of singularities of such orbit closures, it is possible to understand if they are normal, what kind of singularities they have, if they are Cohen-Macaulay or Gorenstein affine varieties.
 
 In this paper we are interested in studying properties of orbit closures inside quiver representations. In particular, we find crepant resolutions of singularities of such orbit closures. Apart from the fact that admitting a crepant resolution has strong geometric implications for the variety in question (e.g. it is Gorenstein, see \cite{BFMTdue}), we are interested in crepancy because it is a very useful property when dealing with orbital degeneracy loci (defined in \cite{BFMT}). Indeed, it allows to explicitly compute the canonical bundle of the orbital degeneracy locus, and therefore can be used to produce new Fano varieties and varieties with trivial canonical bundle.
 
 We chose to analyze the case of quiver orbits because of the nice properties they have. In particular, when the quivers are of finite type, a complete description of such orbits is known (\cite{Gabriel}). We dealt with the quivers of type $A_n$ and $D_4$. 
 
 After recalling basic facts about orbital degeneracy loci and quiver representations, we explain Reineke's construction of resolutions of singularities of quiver orbit closures. Then we use those resolutions to find crepant ones for quivers of type $A_3$, one-way and source-sink quivers of type $A_n$, and a quiver of type $D_4$. Finally, some orbital degeneracy loci fourfolds (and one threefold) with trivial canonical bundle are constructed using the results in the paper.  

\begin{ack}
This work has been carried out in the framework of the Labex Archim\`ede (ANR-11-LABX-0033) and 
of the A*MIDEX project (ANR-11-IDEX-0001-02), funded by the ``Investissements d'Avenir" 
French Government programme managed by the French National Research Agency.
\end{ack}

\section{Quiver Representations}

Let $Q=(S,A)$ be a quiver; $S$ is the set of vertices of the quiver, and $A$ the set of arrows. Each arrow $a\in A$ starts from the vertex $a(0)\in S$ and ends on the vertex $a(1)\in S$. A quiver is said to be \emph{source-sink} if for each vertex $s\in S$, either all the arrows that are connected to $s$ start from it, or they end on it.

If $T$ is the type of a certain semisimple Lie algebra (e.g. $T=A_n, B_n, \dots$) then we will say that the quiver $Q$ is of type $T$ if the underlying graph of $Q$ (which is obtained by replacing arrows with edges) is the Dynkin diagram of type $T$. We will always assume that the quiver has no loops, meaning that its graph has no loops, so that the possible Dynkin types are $A_n, D_n, E_6, E_7, E_8$. Actually, for the definition of quivers of type $B_n, C_n, G_2, F_4$, one needs a more general definition of quiver (i.e. that of a \emph{valued quiver}, as is defined for example in \cite{DlabRingel}). 

If $Q$ is of type $A_n$, we will say that it is a \emph{one-way} quiver if all the arrows point in the same direction. 

\begin{definition}
A representation $\phi$ of a quiver $Q=(S,A)$ is defined by the choice of a vector space $V_s$ of dimension $d_s$ for all $s\in S$ and of a morphism $\phi_a: V_{a(0)}\to V_{a(1)}$ for every $a\in A$. The data $d=\{d_s\}_{s\in S}$ is usually referred to as the dimension vector of the representation. 
\end{definition}

\begin{definition}
Consider two representations $\phi=((V_s)_{s\in S}, (\phi_a)_{a\in A})$ and $\psi=((W_s)_{s\in S}, (\psi_a)_{a\in A})$ of a quiver $Q$. A morphism $\alpha:\phi \to \psi$ between them is the data of morphisms $\alpha_s: V_s \to W_s$ for every $s\in S$ such that for every $a\in A$ the following diagram commutes:
\[
\xymatrix @C=2pc @R=0.4pc{
V_{a(0)} \ar[dd]^-{\alpha_{a(0)}} \ar[r]^-{\phi_a} & \rule{1pt}{0pt}V_{a(1)} \ar[dd]^-{\alpha_{a(1)}} \\
& &\\
W_{a(0)} \ar[r]^-{\psi_a}& \rule{1pt}{0pt}W_{a(1)} 
}
\]
\end{definition}

The parameter space ${\mathcal R}_d$ for representations of a quiver $Q$ with a given dimension vector $d$ (and fixed spaces $(V_s)_{s\in S}$) is
\[
{\mathcal R}_d:=\bigoplus_{a\in A}\Hom(V_{a(0)}, V_{a(1)})\,\,,
\]
on which there is a natural action of 
\[
G:= \prod_{s\in S}\GL(V_s) \,\,.
\]
\begin{remark}
One can look at ${\mathcal R}_d$ as a representation of the group $G$. It is straightforward to see that orbits in ${\mathcal R}_d$ under the action of $G$ are in one-to-one correspondence with isomorphism classes of representations with dimension vector $d$.
\end{remark}

We address now the problem of classifying the orbits in such a parameter space ${\mathcal R}_d$, which is equivalent to classifying isomorphism classes of representations of $Q$. For some special type of quivers, Gabriel's theorem (\cite{Gabriel}) gives this classification:

\begin{definition}
A quiver $Q$ is said to be of finite type if there exists only a finite number of isomorphism classes of indecomposable representations of $Q$.
\end{definition}

\begin{theorem}[\cite{Gabriel}]
\label{Gabriel's theorem}
A quiver $Q$ is of finite type if and only if $Q$ is of type $A_n$, $D_n$, $E_6$, $E_7$, $E_8$ or a finite disjoint unions of those. Moreover in this case, there is a bijection between isomorphism classes of indecomposable representations and positive roots of the Dynkin diagram underlying $Q$.
\end{theorem}

\begin{corollary}
If $Q$ is of finite type, there are only finitely many orbits in ${\mathcal R}_d$ under the action of $G$ for a fixed $d$.
\end{corollary}

\begin{example}
\label{ex.matrices}
Consider a quiver of type $A_2$ (with the arrow toward the second vertex, for example). The parameter space for representations ${\mathcal R}_d={\mathcal R}_{d_1, d_2}$ is the space of matrices $\Hom(V_1, V_2)$, where $\dim(V_i)=d_i$, $i=1,2$. The indecomposable representations are three; we denote by $\alpha_i$ the $i$-th simple root of $A_2$, and the indecomposable representations associated to the positive roots as $\phi_{\alpha_1}$, $\phi_{\alpha_2}$, $\phi_{\alpha_1 + \alpha_2}$. Their dimension vectors are respectively $(1,0)$, $(0,1)$, $(1,1)$. Notice that these vectors are given by the coefficients of the positive root in the basis of simple roots; this is actually a general fact which comes from a more precise statement of Gabriel's theorem.

Orbits in ${\mathcal R}_{d_1, d_2}$ correspond to the possible direct sums of the indecomposable representations with dimension vector $(d_1,d_2)$. For $0\leq r\leq min(d_1,d_2)$, they are of the form $r \phi_{\alpha_1+\alpha_2}\oplus (d_1-r)\phi_{\alpha_1}\oplus (d_2-r)\phi_{\alpha_2}$. The corresponding orbit $\cO_r$ is the orbit in $\Hom(V_1, V_2)$ under the action of $\GL(V_1)\times \GL(V_2)$ of matrices of rank $r$.
\end{example}

%
%
%
%

\subsection{Reineke's resolutions}

In the following we study the orbit closures in ${\mathcal R}_d$ for a fixed quiver $Q$ under the action of $G=\prod_{s\in S}\GL(V_s)$. More precisely, we give a description of a desingularization of the orbit closures, which was first found by Reineke in \cite{Reineke}. In Section \ref{sec.Examples}, we will see some concrete examples of those desingularizations. The power of Reineke's construction is that it gives a desingularization of each orbit closure; by working out the examples, we will see that in some cases these desingularizations are easy to understand and very natural.

We will follow Reineke's paper \cite{Reineke}. The desingularizations are indexed by monomials in $S$. 
\begin{definition}
A monomial in $S$ is a couple $(\vec{s},\vec{a})$, where $\vec{s}=(s_1,\dots ,s_\tau)\in S^\tau$, $\vec{a}=(a_1,\dots ,a_\tau)\in \mathbb{N}^\tau$ for a certain $\tau$.
\end{definition}
To each monomial we will associate a flag variety $F_{(\vec{s},\vec{a})}$ and a vector bundle $W_{(\vec{s},\vec{a})}$ on it. Consider the vector space $V=\oplus_{s\in S} V_s$. For $s\in S$, a subspace (or more generally, a quotient of a subspace) of $V$ is said to be pure of type $s$ if it is contained in $V_s$ (if it is generated by a subspace of $V_s$). Then
\[
F_{(\vec{s},\vec{a})} = \Bigg\{
\begin{array}{c} \mbox{Flags } 0=F^\tau\subset \dots \subset F^1 \subset F^0=V \mbox{ s.t. } \\
F^{i-1}/F^i \mbox{ is pure of type }s_i \mbox{ and dimension }a_i 
\end{array}
\Bigg\}
\]
In fact, it is possible to write this variety as a product of usual flag varieties: it suffices to separate the contributions of the different $V_s$'s. For $s\in S$, let us define $(a^s_1,\dots, a^s_\mu)=(a_{i_1},\dots, a_{i_\mu})$, where $(i_1,\dots,i_\mu)$ are all the occurrences of $s$ in $\vec{s}$ (i.e. $s_{i_1}=\cdots =s_{i_\mu} =s$). Then :
\[
F_{(\vec{s},\vec{a})} \cong \prod_{s\in S} \{\mbox{Flags in } V_s \mbox{ whose i-th successive quotient has dimension }a^s_i\} =
\]
\[
= \prod_{s\in S} F(a^s_1, a^s_1+a^s_2, \dots, V_s)
\]

The vector bundle $W_{(\vec{s},\vec{a})}$ is given by the elements in ${\mathcal R}_d$ which are compatible with the flag variety $F_{(\vec{s},\vec{a})}$; more precisely, over a point 
\[
F^{\bullet}=\{0=F^\tau\subset \dots \subset F^1 \subset F^0=V\} \in F_{(\vec{s},\vec{a})},
\]
the fiber of this bundle is defined as:
\[
(W_{(\vec{s},\vec{a})})_{F^{\bullet}}=\{ w\in {\mathcal R}_d\subset End(V) \mbox{ s.t. } w(F^k)\subset F^k \mbox{ for }0\leq k \leq \tau\}.
\]
We denote by $\pi_{(\vec{s},\vec{a})}: W_{(\vec{s},\vec{a})} \to {\mathcal R}_d$ the natural projection.
\begin{remark}
By the description of $F_{(\vec{s},\vec{a})}$ as a product of flag varieties, it is clear that it is a homogeneous variety under the action of $G$, and the vector bundle $W_{(\vec{s},\vec{a})}$ is $G$-homogeneous. Being the projection morphism $\pi_{(\vec{s},\vec{a})}$ $G$-equivariant, its image is $G$-stable. Moreover, it is irreducible and closed (since $\pi_{(\vec{s},\vec{a})}$ is projective). If $Q$ is of finite type, there is a finite number of orbits in ${\mathcal R}_d$, and therefore $\pi_{(\vec{s},\vec{a})}(W_{(\vec{s},\vec{a})} )$ is the closure of one of them. Actually, for every orbit closure in ${\mathcal R}_d$, a resolution of singularities of the type $W_{(\vec{s},\vec{a})}$ for a certain monomial $(\vec{s},\vec{a})$ always exists (and is not unique!), as explained by the following theorem.
\end{remark}

\begin{theorem}[\cite{Reineke}]
For each representation $M\in {\mathcal R}_d$, there exists a monomial $(\vec{s},\vec{a})(M)$ such that $\pi_{(\vec{s},\vec{a})(M)}$ is a desingularization of the orbit closure $\overline{\mathcal O}_M$ of $M$, and is an isomorphism when restricted to the preimage of ${\mathcal O}_M$.
\end{theorem}

The theorem is proven by showing that each representation $M$ admits a filtration (of indecomposable representations) of a certain type $(\vec{s},\vec{a})(M)$, and then constructing from this data a desingularization of $\overline{\cO}_M$.

\begin{example}
\label{ex.res.matrices}
Before considering non-trivial cases, let us convince ourselves that despite the heavy notation, we are dealing with familiar objects. We start from the situation described in Example \ref{ex.matrices}. We want to construct a desingularization of $Y_r=\overline{\cO}_r\subset \Hom(V_1,V_2)$ by using Reineke's method. A monomial that defines a resolution of singularities of $Y_r$ is given by $(s_1, d_1-r)$ (where $s_1$ is the first vertex). Then:
\[
F:=F_{(s_1, d_1-r)}\cong \Gr(d_1-r, V_1)
\]
and, if $P\in F$,
\[
(W_{(s_1, d_1-r)})_P=\{M\in \Hom(V_1,V_2) \mbox{ s.t. } M|_P=0\}
\]
Therefore, $W_{(s_1, d_1-r)}=\cQ^*\otimes V_2$, where $\cQ$ is the quotient bundle over $F$. This is a well known resolution of determinantal variety. Other choices of the monomial are possible, and give other desingularizations of the same orbit closure.
\end{example}

\section{Crepant resolutions}
\label{sec.Examples}
The aim of this section is to find resolutions of orbit closures in quiver representations, and conditions under which they are crepant. We will start with quivers of type $A_3$, and then we will deal with quivers of type $A_n$. Finally, in order to go beyond the $A_n$ case, we will give some results on quivers of type $D_4$.

\subsection{The quiver $A_3$}
Let us consider a quiver $Q=(S,A)$ of type $A_3$, where $S=\{s_1,s_2,s_3\}$, and $A=\{a_1,a_2\}$. We also fix vector spaces $V_1$, $V_2$, $V_3$ of dimension vector $d={d_1,d_2,d_3}$. According to the direction of the arrows, three different quivers occur. We study them separately.

\subsubsection{ Quiver with $a_1(1)=a_2(1)=s_2$}
\label{secA3Sutar}

The quiver is represented by the following picture:

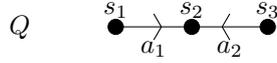
\begin{center}
\begin{tikzpicture}
\draw 
(-1,0) node[anchor=east]  {$Q$};
\draw[fill=black]
(0,0) circle [radius=.1] node [above] {$s_1$}
(1,0) circle [radius=.1] node [above] {$s_2$}
(2,0) circle [radius=.1] node [above] {$s_3$};
\draw 
(0,0) --++ (1,0) 
(1,0) --++ (1,0);    
\draw
(.6,0) --++ (120:.2)
(.6,0) --++ (-120:.2)
(1.4,0) --++ (60:.2)
(1.4,0) --++ (-60:.2);   
\draw
(0.5,-0.3) node {$a_1$}
(1.5,-0.3) node {$a_2$};
\end{tikzpicture}
\end{center}

This configuration has been studied in \cite{Sutar}. ${\mathcal R}_d$ is the representation $\Hom(V_1, V_2)\oplus \Hom(V_3,V_2)$. 
Fix three integers $r_1$, $r_2$, $p_1$. Under the action of $G=\GL(V_1)\times \GL(V_2)\times \GL(V_3)$, 
all the orbits in ${\mathcal R}_d$ 
are of the form
\[
\cO_{r_1, r_2, p_1}=\Bigg\{
\begin{array}{c} \phi \in {\mathcal R}_d \mbox{ s.t. } \dim(\im \phi_{a_1})=r_1 \mbox{ , }\dim(\im \phi_{a_2})=r_2 \mbox{ , } \\
\dim(\im \phi_{a_1} + \im \phi_{a_2})=p_1
\end{array}
\Bigg\}
\]
for all geometrically possible $r_1$, $r_2$, $p_1$ (e.g. $p_1\geq \max\{r_1,r_2\}$). 
This means that what defines an orbit is the rank of the morphisms, and the relative position of the images, i.e. the dimension of the sum (we will see that in the other quivers too, a similar description of the orbits will hold). 

\begin{remark}
\label{notatD4phi}
If needed, for a given element $\phi$ in the orbit, we will denote by $U_i$, $U_{ij}$, $U_{123}$ respectively the image of $\phi_{a_i}$, the intersection $\im \phi_{a_{i}} \cap \im \phi_{a_{j}}$ and the intersection $\im \phi_{a_{1}} \cap \im \phi_{a_{2}}\cap \im \phi_{a_3}$ (the last will make sense for quivers of type $D_4$); their respective dimensions will be denoted by $u_i$, $u_{ij}$, $u_{123}$. Moreover, we will denote by $S_{ij}$, $S_{123}$ respectively the sum $\im \phi_{a_{i}} + \im \phi_{a_{j}}$ and the sum $\im \phi_{a_{1}} + \im \phi_{a_{2}} + \im \phi_{a_3}$ (again, the last will make sense for quivers of type $D_4$), of dimensions $s_{ij}$, $s_{123}$.
\end{remark}

When we consider the closure of such an orbit we obtain:
\begin{lemma}
\label{lemorbclosureA3}
The orbit closure $\overline{\cO}_{r_1, r_2, p_1}$ is given by:
\[
\overline{\cO}_{r_1, r_2, p_1}=\Bigg\{
\begin{array}{c} \phi \in {\mathcal R}_d \mbox{ s.t. } \dim(\im \phi_{a_1})\leq r_1 \mbox{ , }\dim(\im \phi_{a_2})\leq r_2 \mbox{ , } \\
\dim(\im \phi_{a_1} + \im \phi_{a_2})\leq p_1
\end{array}
\Bigg\}
\]
\end{lemma}

\begin{proof}
The inequalities come from the fact that the dimension of the kernels of $\phi_{a_1}$, $\phi_{a_2}$, $\phi_{a_1}\oplus\phi_{a_2}$ can only be greater than those of the elements in the orbit. 

When $u_1+u_2\geq p_1$, the upper bound for $s_{12}$ is $p_1$ because one can suppose that, moving inside $\cO_{r_1, r_2, p_1}$, the images of the two morphisms collapse inside their intersection in a complementary way; if $u_1+u_2< p_1$, the upper bound is reached when the intersection of the images of the two morphisms is zero. 

Moreover, whenever the points $\phi$ with given $u_1,u_2$ and $s_{12}$ belong to the orbit closure, then also the points $\psi$ with $\dim(\im \psi_{a_1})=u_1$, $\dim(\im \psi_{a_1})=u_2$ and $\dim(\im \psi_{a_1} + \im \psi_{a_2})\leq s_{12}$ belong to it, because suitable subspaces of the images of the morphisms $\phi_{a_1}$ and $\phi_{a_2}$ can collapse onto each other, so that the dimension of the intersection raises.
\end{proof}

\begin{remark}
\label{ex.A3}
Recall that each orbit corresponds to an equivalence class of quiver representations. Moreover, equivalence classes of indecomposable representations correspond to positive roots of $A_3$ (Theorem \ref{Gabriel's theorem}), and are indexed by their dimension vectors (see Example \ref{ex.matrices}). Therefore, each representation is of the form
\[
a(0,1,0)\oplus b(1,1,0)\oplus c(0,1,1)\oplus d(1,1,1)\oplus e(1,0,0)\oplus f(0,0,1)
\]
for some integers $a,b,c,d,e,f$. It corresponds to the orbit $\overline{\cO}_{r_1, r_2, p_1}$, with $d_1=b+d+e$, $d_2=a+b+c+d$, $d_3=c+d+f$, $r_1=b+d$, $r_2=c+d$, $p_1=b+c+d$ (a similar interpretation of the orbits holds for the other $A_3$ cases, even though we will not give further details for them).
\end{remark}

In \cite{Sutar}, Sutar used the geometric technique (see \cite{Weyman}) and Reineke's resolutions to classify all Gorenstein orbits:

\begin{theorem}[\cite{Sutar}]
\label{thm.Sutar}
The orbit closure $\overline{\cO}_{r_1, r_2, p_1}$ in ${\mathcal R}_d$ is Gorenstein if and only if one of the following conditions hold:
\begin{itemize}
\item[(i)] $d_1=d_3=p_1$, $d_2=r_1+r_2$;
\item[(ii)] $d_3=r_1=p_1$, $d_2=d_1+r_2$;
\item[(iii)] $r_1=p_1$, $r_2=d_3$, $d_2=d_1+d_3$;
\item[(iv)] $d_1=r_2=p_1$, $d_2=d_3+r_1$;
\item[(v)] $r_2=p_1$, $r_1=d_1$, $d_2=d_1+d_3$.
\end{itemize}
\end{theorem}

\begin{remark}
Conditions $(iv)$ and $(v)$ are equivalent to conditions $(ii)$ and $(iii)$ when the two vertices of $Q$ are exchanged.
\end{remark}

\begin{remark}
Notice that condition $(ii)$ implies that $\im \phi_{a_2} \subset \im \phi_{a_1}$, 
and condition $(iii)$ implies that $\im \phi_{a_2} \subset \im \phi_{a_1}$ and $\Ker \phi_{a_2}=\{0\}$.
\end{remark}

We now use Reineke's construction to find desingularizations for all orbits in ${\cal R}_d$, therefore they will be total spaces of a vector bundle $\cW\subset {\mathcal R}_d \times F$ over a homogeneous variety $F$. In practice, we will distinguish between three types of orbits, each of which admits a different kind of resolution. As we are interested in crepant resolutions, we try to find those resolutions $\pi:Z\to \overline{\cO}$ which are not dominated by some other desingularization $\pi':Z'\to \overline{\cO}$ (i.e. such that there exists no morphism $f:Z' \to Z$ with the property that $\pi'=\pi\circ f$). This condition guided our choice of desingularization, even though there are other possibilities one could consider.

Let us define the variety
\[
F_{i}=\Gr(d_1-r_1, V_1)\times \Gr(p_1, V_2)\times \Gr(d_3-r_2, V_3).
\]
Let us denote $\cU_i$, $\cQ_i$ the tautological and quotient bundle over the Grassmannian of subspaces in $V_i$; then
\[
W_{i}=(\cQ_1^* \oplus \cQ_3^*)\otimes \cU_2\subset (V_1^*\oplus V_3^*)\otimes V_2.
\]
The total space of $\cW_i$ is a desingularization of $\overline{\cO}_{r_1, r_2, p_1}$ when $p_1\neq r_1,r_2$. Indeed, the variety $F_i$ is parametrizing the planes $\Ker \phi_1\subset V_1$, $\Ker \phi_2\subset V_3$ and $\im \phi_1 + \im \phi_2 \subset V_2$ for all the elements $\phi \in {\cO}_{r_1, r_2, p_1}$.

\begin{remark}
This total space is actually one of Reineke's resolutions. In the notation of the previous section, define:
\[
(\vec{s}, \vec{a})=((s_1,s_3,s_2,s_1,s_3),(d_1-r_1,d_3-r_2,p_1,d_1,d_3)).
\]
Then, $(F_{(\vec{s},\vec{a})}, W_{(\vec{s},\vec{a})})$ desingularizes the orbit of elements $\phi\in {\cal R}_d$ such that $\Ker(\phi_{a_1})$ contains a space of dimension $d_1-r_1$, $\Ker(\phi_{a_2})$ contains a space of dimension $d_3-r_2$, $\im(\phi_{a_1})$ and $\im(\phi_{a_2})$ are both contained in the same space of dimension $p_1$. This means that $(F_{(\vec{s},\vec{a})}, W_{(\vec{s},\vec{a})})$ desingularizes $\overline{\cO}_{r_1, r_2, p_1}$, and one can easily see that it is equal to $(F_i, W_i)$.

The next two desingularizations admit the same interpretation, respectively for 
\[
(\vec{s}, \vec{a})=((s_3,s_2,s_1,s_3),(d_3-r_2,r_1,d_1,d_3))
\]
and for
\[
(\vec{s}, \vec{a})=((s_2,s_1,s_3),(r_1,d_1,d_3)).
\]
In the following of the paper, we will not explicit the vectors $(\vec{s}, \vec{a})$ for the resolutions we will use, but it should be kept in mind that they can be interpreted as belonging to Reineke's construction.
\end{remark}

In the case $p_1=r_1$ we consider
\[
F_{ii}=\Gr(r_1,V_2)\times \Gr(d_3-r_2, V_3)
\]
with the bundle
\[
W_{ii}=(V_1^*\oplus \cQ_3^*)\otimes \cU_2
\]
whose total space is a desingularization of the orbit closure $\overline{\cO}_{r_1, r_2, r_1}$ when $r_2\neq d_3$. In this case, as $p_1=r_1$, we do not need to fix the kernel of $\phi_1$ because we are already fixing its image in $V_2$. 

Finally, if $p_1=r_1$ and $r_2= d_3$, the resolution of $\overline{\cO}_{r_1, r_2, p_1}$ is given by the total space of $W_{iii}$ over $F_{iii}$, where:
\[
F_{iii}=\Gr(r_1,V_2),
\]
\[
W_{iii}=(V_1^*\oplus V_2^*)\otimes \cU_2.
\]

\begin{remark}
The resolutions Sutar used in her work are not in general the same as those we chose to use. Indeed, she used Reineke's resolutions $(F_{(\vec{s},\vec{a})}, W_{(\vec{s},\vec{a})})$ such that 
\[
\xi^*=(F_{(\vec{s},\vec{a})} \times {\cal R}_d) /W_{(\vec{s},\vec{a})}=(\cU_1^*\oplus \cU_3^*)\otimes Q_2 .
\]
For these resolutions the expression of $\xi$ in terms of irreducible bundles is very easy. This is particularly useful to apply the geometric technique, for which the cohomology of $\wedge^i \xi$ has to be computed.
\end{remark}

For what concerns the crepant condition for those resolutions, we have the following result:

\begin{proposition}
\label{prop.crepantA3}
Let us take the orbit closure $\overline{\cO}_{r_1, r_2, p_1}\subset {\mathcal R}_d$, and consider its resolution of singularities of the form described above. The orbit closure is Gorenstein if and only if the resolution is crepant. 
\end{proposition}

\begin{proof}
Let us suppose the resolution is given by $W_i$ over $F_i$; we prove that it is crepant when the orbit satisfies condition $(i)$ of Theorem \ref{thm.Sutar}. The proof for $W_{ii}$ over $F_{ii}$ (condition $(ii)$) and $W_{iii}$ over $F_{iii}$ (condition $(iii)$) is similar.

The crepancy condition of a Kempf collapsing gives in our case $\det(W_i)=K_{F_i}$. We have
\[
\det(W_i)=\cO_1(-p_1)\otimes \cO_2(-r_1-r_2)\otimes \cO_3(-p_1)
\]
and
\[
K_{F_i}=\cO_1(-d_1)\otimes \cO_2(-d_2)\otimes \cO_3(-d_3).
\]
By equating the different terms, we get condition $(i)$.
\end{proof}

\begin{remark}
The resolution given by $W_{iii}$ over $F_{iii}$ is actually the resolution of determinantal varieties for $\Hom(V_1\oplus V_3, V_2)$ of rank $r_1$. Therefore, the corresponding orbit is a determinantal orbit. Indeed, the Gorenstein condition is the same as the one that holds for determinantal varieties ($d_1+d_3=d_2$).
\end{remark}

\subsubsection{ Quiver with $a_1(0)=a_2(0)=s_2$}
\label{secA3tipo2}

This quiver is represented by the following picture:

\begin{center}
\begin{tikzpicture}
\draw 
(-1,0) node[anchor=east]  {$Q$};
\draw[fill=black]
(0,0) circle [radius=.1] node [above] {$s_1$}
(1,0) circle [radius=.1] node [above] {$s_2$}
(2,0) circle [radius=.1] node [above] {$s_3$};
\draw 
(0,0) --++ (1,0) 
(1,0) --++ (1,0);    
\draw
(.4,0) --++ (60:.2)
(.4,0) --++ (-60:.2)
(1.6,0) --++ (120:.2)
(1.6,0) --++ (-120:.2);   
\draw
(0.5,-0.3) node {$a_1$}
(1.5,-0.3) node {$a_2$};
\end{tikzpicture}
\end{center}

The representation we are dealing with is ${\mathcal R}_d=\Hom(V_2,V_1)\oplus \Hom(V_2,V_3)$. Fix three integers $k_1$, $k_2$, $q_1$. Then the orbits are
\[
\cO_{k_1, k_2, q_1}=\Bigg\{
\begin{array}{c} \phi \in {\mathcal R}_d \mbox{ s.t. } \dim(\Ker \phi_{a_1})\geq k_1 \mbox{ , }\dim(\Ker \phi_{a_2})= k_2 \mbox{ , } \\
\dim(\Ker \phi_{a_1} \cap \Ker \phi_{a_2})= q_1
\end{array}
\Bigg\}
\]
for $d_2-d_1 \leq k_1 \leq d_2$, $d_2-d_3 \leq k_2\leq d_2$, and $q_1\leq \min\{k_1,k_2\}$. Again, what matters is the relative position of the two subspaces (kernels) in $V_2$.
\begin{lemma}
The orbit closure $\overline{\cO}_{k_1, k_2, q_1}$ is given by:
\[
\overline{\cO}_{k_1, k_2, q_1}=\Bigg\{
\begin{array}{c} \phi \in {\mathcal R}_d \mbox{ s.t. } \dim(\Ker \phi_{a_1})\geq k_1 \mbox{ , }\dim(\Ker \phi_{a_2})\geq k_2 \mbox{ , } \\
\dim(\Ker \phi_{a_1} \cap \Ker \phi_{a_2})\geq q_1
\end{array}
\Bigg\}
\]
\end{lemma}
\begin{proof}
Consider the dual situation, i.e. dual morphisms $\phi_{a_1}^*$ and $\phi_{a_2}^*$, and reason as in the proof of Lemma \ref{lemorbclosureA3}.
\end{proof}

Depending on the choice of $k_1$, $k_2$, $q_1$, there are two different kinds of resolutions.
On one hand, if $q_1\neq k_1,k_2$, then we have resolution $(i)$ given by the total space of
\[
W_i:=\cQ_2^* \otimes (\cU_1\oplus \cU_3) \,\,\, \mbox{ over }\,\,\,
F_i:= \Gr(d_2-k_1, V_1)\times \Gr(q_1,V_2)\times \Gr(d_2-k_2,V_3).
\]
On the other hand, if $q_1=k_1$ (and similarly for $q_1=k_2$), we have resolution $(ii)$ given by
\[
W_{ii}:=\cQ_2^* \otimes (V_1\oplus \cU_3) \,\,\, \mbox{ over }\,\,\,
F_{ii}:= \Gr(q_1,V_2)\times \Gr(d_2-k_2,V_3).
\]

The following proposition describes when resolutions of type $(i)$ and $(ii)$ are crepant.

\begin{proposition}
\begin{itemize}[leftmargin=3.5ex]
	\item[] Resolutions of type $(i)$ are crepant when $d_1=d_3=d_2-q_1$, $k_1+k_2=d_2$.
	\item[] Resolutions of type $(ii)$ are crepant when $q_1=k_1$, $d_1=k_1$, $d_3=d_2-k_1$.
\end{itemize}
\end{proposition}

\begin{proof}
The proof is similar to the one of Proposition \ref{prop.crepantA3}.
\end{proof}

\begin{remark}
The quiver studied in this section can be thought of as being the \emph{dual} of the one considered in Section \ref{secA3Sutar}. Indeed, the representations (orbits, desingularizations) of one can be obtained from the other by considering the dual situation, i.e. by passing from morphisms $\phi_{a_1}$ and $\phi_{a_2}$ to their duals $\phi_{a_1}^*$ and $\phi_{a_2}^*$. The same can be said for the quivers that will be studied in Section \ref{secAtypeI} and Section \ref{secAtypeII}.
\end{remark}

\subsubsection{ Quiver with $a_1(0)=s_1$, $a_2(0)=s_2$}
\label{secA3tipo3}

This quiver is represented by the following picture:

\begin{center}
\begin{tikzpicture}
\draw 
(-1,0) node[anchor=east]  {$Q$};
\draw[fill=black]
(0,0) circle [radius=.1] node [above] {$s_1$}
(1,0) circle [radius=.1] node [above] {$s_2$}
(2,0) circle [radius=.1] node [above] {$s_3$};
\draw 
(0,0) --++ (1,0) 
(1,0) --++ (1,0);    
\draw
(.6,0) --++ (120:.2)
(.6,0) --++ (-120:.2)
(1.6,0) --++ (120:.2)
(1.6,0) --++ (-120:.2);   
\draw
(0.5,-0.3) node {$a_1$}
(1.5,-0.3) node {$a_2$};
\end{tikzpicture}
\end{center}

The representation we are dealing with is ${\mathcal R}_d=\Hom(V_1,V_2)\oplus \Hom(V_2,V_3)$. Fix three integers $r_1$, $k_2$, $u_1$. Then, the orbits are
\[
\cO_{r_1, k_2, u_1}=\Bigg\{
\begin{array}{c} \phi \in {\mathcal R}_d \mbox{ s.t. } \dim(\im \phi_{a_1})= r_1 \mbox{ , }\dim(\Ker \phi_{a_2})= k_2 \mbox{ , } \\
\dim(\im \phi_{a_1} + \Ker \phi_{a_2})= u_1
\end{array}
\Bigg\}
\]
for $r_1\leq \min\{d_2,d_1\}$, $d_2-d_3\leq k_2\leq d_2$, and $\min\{r_1,k_2\}\leq u_1\leq \min\{r_1+k_2,d_2\}$.
Again, what matters is the relative position of the two subspaces (image and kernel) in $V_2$. 

\begin{lemma}
\[
\overline{\cO}_{r_1, k_2, u_1}=\Bigg\{
\begin{array}{c} \phi \in {\mathcal R}_d \mbox{ s.t. } \dim(\im \phi_{a_1})\leq r_1 \mbox{ , }\dim(\Ker \phi_{a_2})\geq k_2 \mbox{ , } \\
\dim(\im \phi_{a_1} + \Ker \phi_{a_2})\leq u_1 + \dim(\Ker \phi_{a_2})-k_2
\end{array}
\Bigg\}
\]
\end{lemma}

\begin{proof}
The only non trivial condition is the last one.  The upper bound of the inequality is attained, for fixed $\phi_{a_1}$, by collapsing the morphism $\phi_{a_2}$ so that its kernel becomes bigger while $\im \phi_{a_1} \cap \Ker \phi_{a_2}$ remains unchanged. On the other hand, when $\im \phi_{a_1}\not\subset \Ker \phi_{a_2}$, the dimension of $\im \phi_{a_1}$ can be reduced without changing $\im \phi_{a_1} + \Ker \phi_{a_2}$. Then, if the dimensions of $\Ker \phi_{a_2}$ and $\im \phi_{a_1}$ are fixed, one can collapse these spaces onto each other; this ``movement" just reduces the dimension of $\im \phi_{a_1} + \Ker \phi_{a_2}$. Hence all the possibilities which satisfy the inequality correspond to points in the orbit closure.
\end{proof}

Depending on the choice of $r_1$, $k_2$, $u_1$, there are two different kinds of resolutions.
On one hand, if $u_1\neq k_2$, then we have resolution $(i)$ given by the total space of
\[
W_i := (\cQ_1^* \otimes \cU_{2,2}) \oplus ((V_2/\cU_{2,1})^*\otimes V_3) \,\,\, \mbox{ over }\,\,\,
F_i:= \Gr(d_1-r_1, V_1)\times F(k_2,u_1,V_2).
\]
where $F(k_2,u_1,V_2)$ is the flag variety with tautological bundles $\cU_{2,1}\subset \cU_{2,2}$ of rank $k_2$, $u_1$.
On the other hand, if $u_1=k_2$, we have resolution $(ii)$ given by
\[
W_{ii}:=(\cQ_1^* \otimes \cU_2) \oplus (\cQ_2^*\otimes V_3) \,\,\, \mbox{ over }\,\,\,
F_{ii}:= \Gr(d_1-r_1,V_1)\times \Gr(k_2,V_2).
\]

\begin{proposition}
\begin{itemize}[leftmargin=3.5ex]
	\item[] Resolutions of type $(i)$ are crepant when $k_2=d_2-r_1=d_1-u_1$, $d_1=d_3$.
	\item[] Resolutions of type $(ii)$ are crepant when $u_1=k_2$, $d_1-r_1=d_2-k_2$, $d_3=2(d_2-k_2)$.
\end{itemize}
\end{proposition}

\subsection{One-way and source-sink quivers of type $A_n$}

In this section we find crepant resolutions for certain orbit closures in ${\mathcal R}_d$ for a quiver of type $A_n$, in a similar way we have proceeded for the $A_3$ quivers. It is a natural generalization of the results of the previous section.

Let us suppose $Q=(S,A)$, is a quiver of type $A_n$ with vertices $S=\{s_1,\dots,s_n\}$ and arrows $A=\{a_1,\dots,a_{n-1}\}$. The vector spaces $V_1,\dots,V_n$ of dimensions $d=(d_1,\dots , d_n)$ are those appearing in the definition of ${\mathcal R}_d$. For later use, let us denote by $F(\alpha_1,\alpha_2, V_i)$ the flag variety in $V_i$ with tautological bundles $\cU_{i,1}\subset \cU_{i,2}$ of ranks $\alpha_1$, $\alpha_2$ (for the Grassmannian, we will write $\cU_{i,1}$ or $\cU_{i,2}$ for the tautological bundle, according to the symmetries of the formula it will appear in).

\subsubsection{One-way quiver of type $A_n$}

This quiver is represented by the following picture (all the arrows point in the same direction):

\begin{center}
\begin{tikzpicture}
\draw 
(-2,0) node[anchor=east]  {$Q$};
\draw[fill=black]
(-1,0) circle [radius=.1] node [above] {$s_1$}
(0,0) circle [radius=.1] node [above] {$s_2$}
(3,0) circle [radius=.1] node [above] {$s_{n-1}$}
(4,0) circle [radius=.1] node [above] {$s_n$};
\draw
(-1,0) --++ (1,0) 
(3,0) --++ (1,0);  
\draw[dashed]
(0,0) --++ (3,0);
\draw
(-0.4,0) --++ (120:.2)
(-0.4,0) --++ (-120:.2)
(3.6,0) --++ (120:.2)
(3.6,0) --++ (-120:.2);   
\draw
(-0.5,-0.3) node {$a_1$}
(3.5,-0.3) node {$a_{n-1}$};
\end{tikzpicture}
\end{center}

Fix integer vectors $k=(k_1,\dots,k_{n-1})$, $t=(t_2,\dots,t_{n-1})$, and consider the orbits given by
\[
\cO_{k,t}=\Bigg\{
\begin{array}{c} \phi \in {\mathcal R}_d \mbox{ s.t. } \dim(\ker \phi_{a_i})= k_i \mbox{ for }1\leq i\leq n-1 \mbox{ , } \\
\dim(\im \phi_{a_i} \cap \Ker \phi_{a_{i+1}})= t_{i+1} \mbox{ for }1\leq i\leq n-2
\end{array}
\Bigg\}.
\]

\begin{remark}
Not all the orbits in ${\mathcal R}_d$ are of the form $\cO_{k,t}$ for some $k,t$. In order to consider other orbits, one should also fix the dimension of other characteristic subspaces, for example of $\Ker \phi_{i+1}\cap \im(\phi_i \circ \phi_{i-1})$, just to name one. By using notations similar to Example \ref{ex.matrices} and Remark \ref{ex.A3}, the orbits $\cO_{k,t}$ correspond to representations of the form
\[
\bigoplus_i \alpha_i (0,\dots,0,1,0,\dots,0) \oplus \bigoplus_j \beta_j (0,\dots,0,1,1,0,0,\dots,0) \oplus 
\]
\[
\oplus \bigoplus_l \gamma_l (0,\dots,0,0,1,1,1,0,\dots,0).
\]
We study those orbits because the desingularization of their closure is the naive generalization of the ones in the $A_3$ case. A similar argument holds for the other $A_n$ cases.
\end{remark}

Suppose $t_i\neq k_i$ for $2\leq i\leq n-1$. A resolution of singularities of $\overline{\cO}_{k,t}$ is given by the total space of the vector bundle
\[
W=((V_{n-1}/\cU_{n-1,1})^*\otimes V_n) \oplus \bigoplus_{i=1}^{n-1}((V_i/\cU_{i,1})^*\otimes \cU_{i+1,2})
\]
over the variety
\[
F=\Gr(k_1,V_1)\times \prod_{i=2}^{n-1} \times F(k_{i},d_{i-1}-k_{i-1}+k_{i}-t_{i}, V_{i}).
\]

\begin{proposition}
Consider an orbit closure $\overline{\cO}_{k,t}$ with $t_i\neq k_i$ for $2\leq i\leq n-1$. Then, the resolution of singularities defined by $W,F$ as above is crepant if and only if there exists $r$ such that
\[
d_1=d_n=d_i-t_i \mbox{ for }2\leq i\leq n-1 \,\,\,\mbox{ and }\,\,\, r= d_i-k_i \mbox{ for }1\leq i\leq n-1.
\]
\end{proposition}

\subsubsection{Source-sink quiver of type $A_{2m}$}

This quiver is represented by the following picture (all nodes are either a source or a sink for the arrows):

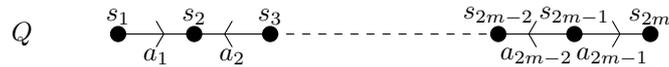
\begin{center}
\begin{tikzpicture}
\draw 
(-3,0) node[anchor=east]  {$Q$};
\draw[fill=black]
(-2,0) circle [radius=.1] node [above] {$s_1$}
(-1,0) circle [radius=.1] node [above] {$s_2$}
(0,0) circle [radius=.1] node [above] {$s_3$}
(3,0) circle [radius=.1] node [above] {$s_{2m-2}$}
(4,0) circle [radius=.1] node [above] {$s_{2m-1}$}
(5,0) circle [radius=.1] node [above] {$s_{2m}$};
\draw
(-2,0) --++ (1,0) 
(-1,0) --++ (1,0) 
(3,0) --++ (1,0) 
(4,0) --++ (1,0);  
\draw[dashed]
(0,0) --++ (3,0);
\draw
(-1.4,0) --++ (120:.2)
(-1.4,0) --++ (-120:.2)
(-0.6,0) --++ (60:.2)
(-0.6,0) --++ (-60:.2)
(3.4,0) --++ (60:.2)
(3.4,0) --++ (-60:.2)
(4.6,0) --++ (120:.2)
(4.6,0) --++ (-120:.2);   
\draw
(-1.5,-0.3) node {$a_1$}
(-0.5,-0.3) node {$a_2$}
(3.5,-0.3) node {$a_{2m-2}$}
(4.5,-0.3) node {$a_{2m-1}$};
\end{tikzpicture}
\end{center}

Fix integer vectors $r=(r_1,\dots,r_{2m-1})$, $p=(p_1,\dots,p_{m-1})$, $q=(q_1,\dots,q_{m-1})$ and consider the orbits given by
\[
\cO_{r,p,q}=\Bigg\{
\begin{array}{c} \phi \in {\mathcal R}_d \mbox{ s.t. } \dim(\im \phi_{a_i})= r_i \mbox{ for }1\leq i\leq 2m-1 \mbox{ , } \\
\dim(\im \phi_{a_{2i-1}} + \im \phi_{a_{2i}})= p_{i} \mbox{ for }1\leq i\leq m-1 \mbox{ , } \\
\dim(\ker \phi_{a_{2i+1}} \cap \ker \phi_{a_{2i}})= q_{i} \mbox{ for }1\leq i\leq m-1
\end{array}
\Bigg\}.
\]

Suppose $p_i\neq r_{2i-1}, r_{2i}$ and $q_i\neq d_{2i+1}-r_{2i+1},d_{2i+1}-r_{2i}$ for $1\leq i\leq m-1$. A resolution of singularities of $\overline{\cO}_{r,p,q}$ is given by the total space of the vector bundle
\[
W=((V_{2m-1} / \cU_{2m-1,2})^*\otimes V_{2m}) \oplus 
\]
\[
\oplus \bigoplus_{i=1}^{m-1} ( ((V_{2i-1}/\cU_{2i-1,2})^* \otimes \cU_{2i,2}) \oplus ((V_{2i+1}/ \cU_{2i+1,1})^*\otimes \cU_{2i,1}) )
\]
over the variety
\[
F=\Gr(d_1-r_1,V_1)\times \prod_{i=1}^{m-1}(  F(r_{2i}, p_i,V_{2i})\times F(q_i,d_{2i+1}-r_{2i+1}, V_{2i+1})  ).
\]

\begin{proposition}
Consider an orbit closure $\overline{\cO}_{r,p,q}$ with $p_i\neq r_{2i-1}, r_{2i}$ and $q_i\neq d_{2i+1}-r_{2i+1},d_{2i+1}-r_{2i}$ for $1\leq i\leq m-1$. Then, the resolution of singularities defined by $W,F$ as above is crepant if and only if
\[
d_1=d_{2m}=p_i \mbox{ for }1\leq i\leq m-1 \,\,\, \mbox{ , }\,\,\, q_i=d_{2i+1}-d_1 \mbox{ for }1\leq i\leq m-1 \,\,\,\mbox{ and }
\]
\[
r_i=d_{i+1}-r_{i+1} \mbox{ for }1\leq i\leq 2m-2.
\]
\end{proposition}

\subsubsection{Source-sink quiver of type $A_{2m+1}$, type $I$}
\label{secAtypeI}

The quiver is represented by the following picture (all nodes are either a source or a sink for the arrows):

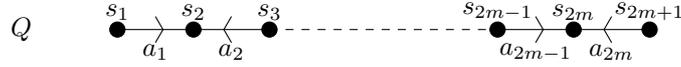
\begin{center}
\begin{tikzpicture}
\draw 
(-3,0) node[anchor=east]  {$Q$};
\draw[fill=black]
(-2,0) circle [radius=.1] node [above] {$s_1$}
(-1,0) circle [radius=.1] node [above] {$s_2$}
(0,0) circle [radius=.1] node [above] {$s_3$}
(3,0) circle [radius=.1] node [above] {$s_{2m-1}$}
(4,0) circle [radius=.1] node [above] {$s_{2m}$}
(5,0) circle [radius=.1] node [above] {$s_{2m+1}$};
\draw
(-2,0) --++ (1,0) 
(-1,0) --++ (1,0) 
(3,0) --++ (1,0) 
(4,0) --++ (1,0);  
\draw[dashed]
(0,0) --++ (3,0);
\draw
(-1.4,0) --++ (120:.2)
(-1.4,0) --++ (-120:.2)
(-0.6,0) --++ (60:.2)
(-0.6,0) --++ (-60:.2)
(3.6,0) --++ (120:.2)
(3.6,0) --++ (-120:.2)
(4.4,0) --++ (60:.2)
(4.4,0) --++ (-60:.2);   
\draw
(-1.5,-0.3) node {$a_1$}
(-0.5,-0.3) node {$a_2$}
(3.5,-0.3) node {$a_{2m-1}$}
(4.5,-0.3) node {$a_{2m}$};
\end{tikzpicture}
\end{center}

Fix integer vectors $r=(r_1,\dots,r_{2m})$, $p=(p_1,\dots,p_{m})$, $q=(q_1,\dots,q_{m-1})$ and consider the orbits given by
\[
\cO_{r,p,q}=\Bigg\{
\begin{array}{c} \phi \in {\mathcal R}_d \mbox{ s.t. } \dim(\im \phi_{a_i})= r_i \mbox{ for }1\leq i\leq 2m \mbox{ , } \\
\dim(\im \phi_{a_{2i-1}} + \im \phi_{a_{2i}})= p_{i} \mbox{ for }1\leq i\leq m \mbox{ , } \\
\dim(\ker \phi_{a_{2i+1}} \cap \ker \phi_{a_{2i}})= q_{i} \mbox{ for }1\leq i\leq m-1
\end{array}
\Bigg\}.
\]

Suppose $p_i\neq r_{2i-1}, r_{2i}$ and $q_j\neq d_{2j+1}-r_{2j+1},d_{2j+1}-r_{2j}$ for $1\leq i\leq m$ and $1\leq j\leq m-1$. A resolution of singularities of $\overline{\cO}_{r,p,q}$ is given by the total space of the vector bundle
\[
W=\bigoplus_{i=1}^{m-1} ( ((V_{2i-1}/\cU_{2i-1,2})^* \otimes \cU_{2i,2}) \oplus ((V_{2i+1}/ \cU_{2i+1,1})^*\otimes \cU_{2i,1}) ) \oplus 
\]
\[
\oplus ((V_{2m-1} / \cU_{2m-1,2})^*\otimes \cU_{2m,2})\oplus (V_{2m+1}^*\otimes \cU_{2m,1})
\]
over the variety
\[
F=\Gr(d_1-r_1,V_1)\times F(r_{2m}, p_m, V_{2m})\times 
\]
\[
\times \prod_{i=1}^{m-1}(  F(r_{2i}, p_i,V_{2i})\times F(q_i,d_{2i+1}-r_{2i+1}, V_{2i+1})  ).
\]

\begin{proposition}
Consider an orbit closure $\overline{\cO}_{r,p,q}$ with $p_i\neq r_{2i-1}, r_{2i}$ and $q_j\neq d_{2j+1}-r_{2j+1},d_{2j+1}-r_{2j}$ for $1\leq i\leq m$ and $1\leq j\leq m-1$. Then, the resolution of singularities defined by $W,F$ as above is crepant if and only if
\[
d_1=d_{2m+1}=p_i \mbox{ for }1\leq i\leq m \,\,\, \mbox{ , }\,\,\, q_i=d_{2i+1}-d_1 \mbox{ for }1\leq i\leq m-1 \,\,\,\mbox{ and }
\]
\[
r_i=d_{i+1}-r_{i+1} \mbox{ for }1\leq i\leq 2m-1.
\]
\end{proposition}

\subsubsection{Source-sink quiver of type $A_{2m+1}$, type $II$}
\label{secAtypeII}

The quiver is represented by the following picture (all nodes are either a source or a sink for the arrows):

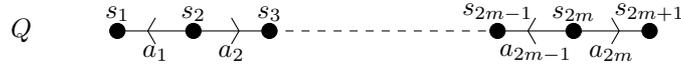
\begin{center}
\begin{tikzpicture}
\draw 
(-3,0) node[anchor=east]  {$Q$};
\draw[fill=black]
(-2,0) circle [radius=.1] node [above] {$s_1$}
(-1,0) circle [radius=.1] node [above] {$s_2$}
(0,0) circle [radius=.1] node [above] {$s_3$}
(3,0) circle [radius=.1] node [above] {$s_{2m-1}$}
(4,0) circle [radius=.1] node [above] {$s_{2m}$}
(5,0) circle [radius=.1] node [above] {$s_{2m+1}$};
\draw
(-2,0) --++ (1,0) 
(-1,0) --++ (1,0) 
(3,0) --++ (1,0) 
(4,0) --++ (1,0);  
\draw[dashed]
(0,0) --++ (3,0);
\draw
(-1.6,0) --++ (60:.2)
(-1.6,0) --++ (-60:.2)
(-0.4,0) --++ (120:.2)
(-0.4,0) --++ (-120:.2)
(3.4,0) --++ (60:.2)
(3.4,0) --++ (-60:.2)
(4.6,0) --++ (120:.2)
(4.6,0) --++ (-120:.2);   
\draw
(-1.5,-0.3) node {$a_1$}
(-0.5,-0.3) node {$a_2$}
(3.5,-0.3) node {$a_{2m-1}$}
(4.5,-0.3) node {$a_{2m}$};
\end{tikzpicture}
\end{center}

Fix integer vectors $r=(r_1,\dots,r_{2m})$, $p=(p_1,\dots,p_{m-1})$, $q=(q_1,\dots,q_{m})$ and consider the orbits given by
\[
\cO_{r,p,q}=\Bigg\{
\begin{array}{c} \phi \in {\mathcal R}_d \mbox{ s.t. } \dim(\im \phi_{a_i})= r_i \mbox{ for }1\leq i\leq 2m \mbox{ , } \\
\dim(\im \phi_{a_{2i}} + \im \phi_{a_{2i+1}})= p_{i} \mbox{ for }1\leq i\leq m-1 \mbox{ , } \\
\dim(\ker \phi_{a_{2i}} \cap \ker \phi_{a_{2i-1}})= q_{i} \mbox{ for }1\leq i\leq m
\end{array}
\Bigg\}.
\]

Suppose $p_i\neq r_{2i+1}, r_{2i}$ and $q_i\neq d_{2i}-r_{2i-1},d_{2i}-r_{2i}$ for all possible $i$. A resolution of singularities of $\overline{\cO}_{r,p,q}$ is given by the total space of the vector bundle
\[
W=\bigoplus_{i=1}^{m-1} ( ((V_{2i}/\cU_{2i,1})^* \otimes \cU_{2i-1,1}) \oplus ((V_{2i}/ \cU_{2i,2})^*\otimes \cU_{2i+1,2}) ) \oplus 
\]
\[
\oplus ((V_{2m} / \cU_{2m,1})^*\otimes \cU_{2m-1,1})\oplus ((V_{2m}/\cU_{2m,2})^*\otimes \cU_{2m+1,2}) 
\]
over the variety
\[
F=\Gr(r_1,V_1)\times F(q_{m}, d_{2m}-r_{2m}, V_{2m})\times 
\]
\[
\times \prod_{i=1}^{m-1}(  F(q_{i}, d_{2i}-r_{2i},V_{2i})\times F(r_{2i+1},p_{i}, V_{2i+1})  ).
\]

\begin{proposition}
Consider an orbit closure $\overline{\cO}_{r,p,q}$ with $p_i\neq r_{2i+1}, r_{2i}$ and $q_i\neq d_{2i}-r_{2i-1},d_{2i}-r_{2i}$ for all possible $i$. Then, the resolution of singularities defined by $W,F$ as above is crepant if and only if
\[
d_1=d_{2m+1}=p_i \mbox{ for }1\leq i\leq m-1 \,\,\, \mbox{ , }\,\,\, q_i=d_{2i}-d_1 \mbox{ for }1\leq i\leq m \,\,\,\mbox{ and }
\]
\[
r_i=d_{i+1}-r_{i+1} \mbox{ for }1\leq i\leq 2m-1.
\]
\end{proposition}

\subsection{The quiver $D_4$}
\label{secD4}

The study of quivers of type $D_4$ presents some interesting difficulties, especially involving the construction of desingularizations for the orbit closures. We will study the crepancy condition for the resolutions of singularities we will consider among Reineke's resolutions.

Let us begin with a quiver $Q(S,A)$ of type $D_4$ with $S=(s_1,s_2,s_3,s_4)$, $A=(a_1,a_2,a_3)$. We will study the quiver with all the arrows pointing toward the central vertex $s_2$ (this is the analogue of the quiver of type $A_3$ studied by Sutar). The quiver is represented by the following picture:

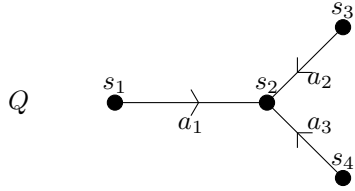
\begin{center}
\begin{tikzpicture}
\draw 
(-3,0) node[anchor=east]  {$Q$};
\draw[fill=black]
(-2,0) circle [radius=.1] node [above] {$s_1$}
(0,0) circle [radius=.1] node [above] {$s_2$}
(1,1) circle [radius=.1] node [above] {$s_3$}
(1,-1) circle [radius=.1] node [above] {$s_4$};
\draw 
(-2,0) --++ (2,0) 
(0,0) --++ (45:1.34)
(0,0) --++ (-45:1.34);    
\draw
(-.9,0) --++ (120:.2)
(-.9,0) --++ (-120:.2)
(0.4,0.4) --++ (90:.2)
(0.4,0.4) --++ (0:.2)
(0.4,-0.4) --++ (-90:.2)
(0.4,-0.4) --++ (0:.2);   
\draw
(-1,-0.3) node {$a_1$}
(0.7,0.3) node {$a_2$}
(0.7,-0.3) node {$a_3$};
\end{tikzpicture}
\end{center}

The vector spaces appearing in the definition of ${\mathcal R}_d$ are $V_1,\dots,V_4$ of dimensions $d_1,\dots,d_4$. Fix integers $r_i$ for $i=1,2,3$, $r_{ij}$ for $1\leq i<j\leq3$, and $r_{123}$. Define $x:=\sum_i r_i -\sum_{i<j}r_{ij}+r_{123}$. All the orbits are of the following form, for geometrically possible $r$:
\[
\cO_{r}=\Bigg\{
\begin{array}{c} \phi \in {\mathcal R}_d \mbox{ s.t. } \dim(\im \phi_{a_i})= r_i \mbox{ for }i=1,2,3 \mbox{ , } \\
\dim(\im \phi_{a_{i}} \cap \im \phi_{a_{j}})= r_{ij} \mbox{ for }1\leq i<j\leq 3 \mbox{ , } \\
\dim(\im \phi_{a_{1}} \cap \im \phi_{a_{2}}\cap \im \phi_{a_3})= r_{123}
\end{array}
\Bigg\}.
\]

In words, we are fixing the relative position of the images of the morphisms.

Suppose that, for the elements in the orbit $\phi\in \cO_r$, there is no equality between the spaces $U_i, U_{ij}, U_{123}$. Then, we consider the Kempf collapsing $(i)$ given by the vector bundle
\[
W_i=((\bigoplus_{i=1,3,4} \cU_{i,2}/\cU_{i,1})^*\otimes \cU_{2,1})\oplus ((\bigoplus_{i=1,3,4}^3 V_i/\cU_{i,2})^*\otimes (V_2 / \cU_{2,1}))
\]
over
\[
F_i=F(d_1-r_1, d_1-r_1+r_{12}+r_{13}-r_{123},V_1)\times \Gr(r_{12}+r_{13}+r_{23}-2r_{123},V_2)\times
\]
\[
\times F(d_3-r_2, d_3-r_2+r_{12}+r_{23}-r_{123},V_3)\times F(d_4-r_3, d_4-r_3+r_{13}+r_{23}-r_{123},V_4).
\]

The motivation for this choice is that the base variety $F_i$ parametrizes, for the elements inside $\overline{\cO}_r$, the subspaces $\Ker\phi_{a_1}\subset \phi_{a_1}^{-1}(U_{12}+U_{13})\subset V_1$ (and similarly for the other spaces $V_3$, $V_4$) and $U_{12}+U_{13}+U_{23}\subset V_2$. 

\begin{remark}
\label{remdimD4}
Even though it may seem we are losing some information (e.g. not fixing the dimension of $U_{123}$), we are not. For instance, by taking a general point $\phi$ of this resolution, $U_{12}+U_{13}$ and $U_{12}+U_{23}$ are well defined subspaces of dimension $r_{12}+r_{13}-r_{123}$ and $r_{12}+r_{23}-r_{123}$ inside $U_{12}+U_{13}+U_{23}$. Therefore, they intersect in a subspace of dimension at least
\[
r_{12}+r_{13}+r_{23}-r_{123}-(r_{12}+r_{13}-r_{123})-(r_{12}+r_{23}-r_{123})=r_{12},
\]
which is exactly the bound we want for the dimension of $U_{12}$. In the same way, the dimension of the intersection of $U_{12}+U_{13}$, $U_{12}+U_{13}$ and $U_{12}+U_{13}$ is at least
\[
2(r_{12}+r_{13}+r_{23}-r_{123})-(r_{12}+r_{13}-r_{123})+
\]
\[
-(r_{12}+r_{23}-r_{123})-(r_{13}+r_{23}-r_{123})=r_{123},
\]
which is the required bound for the dimension of $U_{123}$. A similar argument will hold for the other desingularizations below.
\end{remark}

Remark \ref{remdimD4} tells us that the image (inside ${\mathcal R}_d$) of this collapsing is contained in $\overline{\cO}_r$. To see if it is a desingularization, let us take an element $\phi\in \cO_r$. Then, its preimage inside $W_i$ lies over the (unique) point of the flag variety $F_i$ whose explicit expression is:
\[
(\Ker \phi_{a_1}\subset \phi_{a_1}^{-1}(U_{12}+U_{13}) \subset V_1) \times (U_{12}+U_{13}+U_{23}\subset V_2) \times
\]
\[
\times (\Ker \phi_{a_2}\subset \phi_{a_2}^{-1}(U_{12}+U_{23}) \subset V_3) \times (\Ker \phi_{a_3}\subset \phi_{a_3}^{-1}(U_{13}+U_{23}) \subset V_4).
\]
Therefore the morphism $W_i\to \overline{\cO}_r$ is generically one-to-one, and as a consequence it makes $W_i$ a desingularization of $\overline{\cO}_r$.

The second resolution $(ii)$ we consider is obtained by just fixing $r_i$ for $i=1,2,3$ and the dimension of $U_1+U_2+U_3$; therefore the dimensions of $U_{ij}$ and $U_{123}$ are the minimal possible for the generic element $\phi\in \cO_r$. Then, the orbit closure $\overline{\cO}_r$ is resolved by the total space of 
\[
W_{ii}=(V_1^*\otimes \cU_{2,1})\oplus (((V_3/\cU_{3,1})\oplus (V_4/ \cU_{4,1}))^*\otimes \cU_{2,2})
\]
over
\[
F_{ii}=F(r_1,x,V_2)\times \Gr(d_3-r_2,V_3)\times \Gr(d_4-r_3,V_4).
\]
This Kempf collapsing is of the same type of the one that could be used for the quivers of type $A_2$, and therefore it should be straightforward to see that it is a desingularization. Notice that this resolution is not symmetric: indeed, a particular role is played by $V_1$, because in $F_{ii}$ we are parametrizing the image of $\phi_{a_1}$ (while for $\phi_{a_2}$, $\phi_{a_3}$ we are parametrizing the kernels). 

Finally, in the third resolution $(iii)$ we fix again $r_i$ for $i=1,2,3$, and the dimension of $U_{1}+U_2+U_3$; as before, the dimensions of $U_{ij}$ and $U_{123}$ are the minimal possible for the generic element $\phi\in \cO_r$. The resolution of the orbit closure is given by 
\[
W_{iii}=((V_1/ \cU_1)\oplus (V_3/ \cU_3)\oplus (V_4/ \cU_4))^*\otimes \cU_{2}
\]
over
\[
F_{iii}=\Gr(d_1-r_1,V_1)\times\Gr(x,V_2)\times\Gr(d_3-r_2,V_3)\times\Gr(d_4-r_3,V_4).
\]
As before, it is straightforward to see that it is a desingularization.

\begin{proposition}
Let $\overline{\cO}_r$ be an orbit closure in ${\mathcal R}_d$ which admits one of the three resolutions $(i)$, $(ii)$, $(iii)$.
\begin{itemize}[leftmargin=3.5ex]
	\item[] The resolution of type $(i)$ is never crepant.
	\item[] The resolution of type $(ii)$ and the resolution of type $(iii)$ are crepant when $d_2=\sum_i r_i$, $d_1=d_3=d_4=x$.
\end{itemize}
\end{proposition}

\begin{proof}
The proof is similar to the one of Proposition \ref{prop.crepantA3}. We just remark that the condition for $(i)$ to be crepant is $d_1=d_2/3=d_3=d_4=r_i=r_{jk}=r_{123}$ for all $i$, $j<k$. But this is not possible, the definition of $F_i$ requires that $r_1<d_1$, $r_2<d_3$, $r_3<d_4$.
\end{proof}

\section{Applications to orbital degeneracy loci}

Orbital degeneracy loci are a generalization of both zero loci and degeneracy loci of morphisms between vector bundles. They can be thought of as the relative construction of a certain \emph{model}, which is typically an orbit closure inside a representation of an algebraic group. In this section, we study quiver degeneracy loci, i.e. degeneracy loci constructed from orbit closures inside quiver representations.

\subsection{Preliminaries}

We recall some notation for orbital degeneracy loci (which will be denoted from now on by ODL). As a reference, see \cite{BFMT}. 

Fix an algebraic group $G$, a finite dimensional $G$-module $V$, and a closed $G$-invariant subvariety $Y\subset V$. Consider a variety $X$ equipped with a $G$-principal bundle $\cE$. One can construct from this data a vector bundle $\cE_V$ on $X$, whose fiber over each point is isomorphic to $V$ as a $G$-module. Let $s$ be a section of $\cE_V$. Then, one can construct the $Y$-degeneracy locus associated to $s$, which is denoted by $D_Y(s)$ (\cite[Definition 2.1]{BFMT}). Indeed, the inclusion $Y\subset V$ can be relativized over $X$ to the inclusion $\cE_Y\subset \cE_V$; then $D_Y(s)$ is defined as the locus of points in $X$ which are sent by $s$ inside $\cE_Y$. If the bundle $\cE_V$ is globally generated, and the section $s$ is general, $D_Y(s)$ satisfies nice properties (\cite[Proposition 2.3]{BFMT}); in particular, $\codim_X (D_Y(s))=\codim_V (Y)$.

Suppose now that there exists a subbundle $W$ of the trivial bundle $V\times G/P$ over the homogeneous variety $G/P$, such that its total space is a desingularization of $Y\subset V$ via the natural projection. Then it is possible to relativize this construction, and obtain a resolution of $D_Y(s)$. This resolution will live inside the variety $\cE_{G/P}$, which admits a fibration $\pi:\cE_{G/P}\to X$ whose fiber is isomorphic to $G/P$. Denote by $Q_W$ the quotient of $\pi^* \cE_V$ by $\cE_W$ (seen as a vector bundle over $\cE_{G/P}$). Then $s$ induces a section $\tilde{s}$ of $Q_W$, and its zero locus $\zero({\tilde{s}})$ is the wanted resolution of $D_Y(s)$ (again refer to \cite[Proposition 2.3]{BFMT}). The following result explains why crepant resolutions are interesting in this context:

\begin{proposition}[\cite{BFMT}]
Suppose that the resolution of singularities $p_W: W\rightarrow Y$ satisfies 
\[
det(W)=K_{G/P}.
\]
If $\cE_V$ is globally generated and $s$ is a general section, then the canonical sheaf of $D_Y(s)$  
is the restriction of some line bundle on $X$. In particular, $D_Y(s)$ is Gorenstein and has canonical singularities.
\end{proposition}

Notice that the line bundle mentioned in the previous proposition can be computed explicitly. This is a consequence of the adjunction formula for zero loci applied to $\zero({\tilde{s}})$.

\subsection{Quiver degeneracy loci}

In this section, we use the results we found on crepancy of resolutions of quiver orbit closures to construct some examples of orbital degeneracy loci. As already pointed out, the fact that the resolution is crepant allows us to compute the canonical bundle of these loci. We exhibit a sample of constructions of varieties (especially fourfolds) with trivial canonical bundle.

All the computations in cohomology, in particular the computation of the Euler characteristic of the trivial bundle, have been done with \Mac (\cite{Macaulay}). With this software in our cases it was possible to explicitly construct the resolution of the orbital degeneracy loci and perform the computations we need in the cohomology ring of the variety.

\subsubsection{Quiver degeneracy loci of type $A_3$}

We begin by considering the case of quivers of type $A_3$ described in Section \ref{secA3Sutar} (refer to it for the notations).
If we want to consider ODLs, we need to fix a smooth projective variety $X$, and three vector bundles $E_1,E_2,E_3$ of dimensions $d_1,d_2,d_3$ on it, such that $\Hom(E_1,E_2) \oplus \Hom(E_3,E_2)$ is globally generated. Then, suppose that $s$ is a general section of this bundle, and fix an orbit closure $Y=\overline{\cO}_{r_1, r_2, p_1}$ inside ${\mathcal R}_d=\Hom(\CC^{d_1},\CC^{d_2})\oplus \Hom(\CC^{d_3},\CC^{d_2})$. Recall that $D_Y(s)$ is the locus of points $x\in X$ which are sent by the section $s$ inside $Y\subset (\Hom(E_1,E_2) \oplus \Hom(E_3,E_2))_x \cong {\mathcal R}_d$.

\begin{theorem}
\label{prop.ODLA3}
Let $D_Y(s)$ be defined as above, where $Y=\overline{\cO}_{r_1, r_2, p_1}\subset {\mathcal R}_d$. Then $K_{D_Y(s)}$, $\codim_X(D_Y(s))$ and a lower bound for $\codim_{D_Y(s)}\Sing(D_Y(s))$ are given by the formulas in Table \ref{tableA3Sutar}. 
\end{theorem}

\begin{table}[h!bt]
	\begin{center}
	\begin{footnotesize}
		\caption{ODL from a quiver of type $A_3$ with $a_1(1)=a_2(1)=s_2$. We use the following variables: $\eta_1=d_1-r_1$, $\eta_2=d_2-p_1$, $\eta_3=d_3-r_2$.}
		\label{tableA3Sutar}
		\begin{tabular}{cccc} \toprule
			$\begin{array}{c}
			\mbox{Case} \\
			\mbox{(Thm. \ref{thm.Sutar}) }
			\end{array}$ & $K_{D_Y(s)/X}$ & $\codim_X D_Y(s)$ & $\codim_{D_Y(s)}\Sing$
			\\ \midrule
			$(i)$ & $\begin{array}{c} 
			\det E_1^{-r_2}\otimes \\
			 \det E_2^{p_1}\otimes \det E_3^{-r_1}
			\end{array}$ & $\begin{array}{c}
			\eta_1^2+\eta_2^2+\eta_3^2+ \\
			+\eta_2(\eta_1+\eta_3)
			\end{array}$ & $\geq \min\Bigg\{\begin{array}{c}
			2\eta_1+1\\
			2\eta_3+1\\
			2\eta_2+1
			\end{array}\Bigg\}$
			\\ \midrule
			$(ii)$ & $\begin{array}{c} 
			\det E_1^{-d_2+p_1}\otimes \\
			 (\det E_2^*\otimes \det E_3)^{-d_2+r_2}
			\end{array}$ & $\eta_3^2+\eta_2^2+\eta_3\eta_2$ & $\geq\min\Bigg\{\begin{array}{c}
			2\eta_3+1\\
			2\eta_2+1
			\end{array}\Bigg\}$
			\\ \midrule
			$(iii)$ & $\begin{array}{c} 
			(\det E_1\otimes \det E_2^*\otimes \\
			 \det E_3)^{-d_2+p_1}
			\end{array}$ & $\eta_2^2$ & $\geq 2\eta_2+1 $		
			\\ \bottomrule
		\end{tabular}
		\end{footnotesize}
	\end{center}
\end{table}

\begin{proof}
By \cite{BFMT}[Prop. 2.3], we have that $\codim_X(D_Y(s))=\codim_{{\mathcal R}_d}Y$ and $\codim_{D_Y(s)}\Sing(D_Y(s))=\codim_Y \Sing(Y)$. The first quantity can be computed directly from the resolution of singularities we have. For the codimension of the singularities, we explain the case of the resolution $W_i$ over $F_i$ (the others are similar). 

The singularities are contained in the locus in the orbit closure $Y$ where the map $\pi: W_i \to Y\subset {\mathcal R}_d$ is not an isomorphism. As $Y$ is normal (\cite{Bobinski}), the morphism $\pi$ restricted to the fiber over such a point is a contraction. Take a point $y\in Y$. If $\dim (\Ker y_{1}) =d_1-r_1$, $dim (\Ker y_{2}) =d_3-r_2$, $\dim (\im y_1 + \im y_2)=p_1$, then $\pi^{-1}y$ is a single point. Therefore, if $y$ is in the singular locus, then either $\dim (\Ker y_{1}) >d_1-r_1$ or $\dim (\Ker y_{2} )>d_3-r_2$ or $\dim (\im y_1 + \im y_2)<p_1$. As a consequence, the singular locus is contained in the union of three orbit closures, i.e.
\begin{equation}
\label{eqsingproof}
\Sing(Y)\subset \overline{\cO}_{r_1-1,r_2,p_1}\cup\overline{\cO}_{r_1,r_2-1,p_1}\cup\overline{\cO}_{r_1,r_2,p_1-1},
\end{equation}
whose dimension are easy to compute from their resolution, giving the bound for $\codim_{D_Y(s)}\Sing$ of Table \ref{tableA3Sutar}.

Finally, to compute the canonical bundle, recall that, from \cite{BFMT}, a crepant resolution of $D_Y(s)$ is given by $\pi|_{\zero(\tilde{s})}:\zero(\tilde{s})\to D_Y(s)$. This is the zero locus of the section $\tilde{s}$ (which is constructed from $s$) of the bundle 
\[
\cQ_{W_i}:= ((E_1^*\oplus E_3^*)\otimes E_2)/ ((\cQ_1^* \oplus \cQ_3^*)\otimes \cU_2)
\]
over the Grassmannian bundle 
\[
F_i(E_1,E_2,E_3):=\Gr(d_1-r_1, E_1)\times \Gr(p_1, E_2)\times \Gr(d_3-r_2, E_3)
\]
over $X$. This, together with the adjunction formula, gives
\[
K_{\zero(\tilde{s})}=(\pi^*(K_X)\otimes K_{F_i(E_1,E_2,E_3)/X}\otimes \det(\cQ_{W_i}))|_{\zero(\tilde{s})}.
\]
An easy computation shows that 
\[
K_{\zero(\tilde{s})}=\pi|_{\zero(\tilde{s})}^*((K_X\otimes \det E_1^{-r_2}\otimes \det E_2^{p_1}\otimes \det E_3^{-r_1})|_{D_Y(s)}),
\]
which implies that
\[
K_{D_Y(s)}=K_X\otimes \det E_1^{-r_2}\otimes \det E_2^{p_1}\otimes \det E_3^{-r_1}.
\]
\end{proof}

\begin{remark}
In order to compute exactly the codimension of the singular locus it would be necessary, for instance, to know that the desingularization morphism does not contract any divisor. In this case, the morphism restricts to an isomorphism over the smooth locus of $Y$, and equality holds in (\ref{eqsingproof}) (thus giving equality in the last column of Table \ref{tableA3Sutar}). 

Suppose that the resolution is given by $W=\cU\otimes W'$ over $F=\Gr(a,v)\times F'$, for $W'$ a vector bundle of rank $w$ over the variety $F'$. Then there is a locus $E$ contracted by the morphism $\pi:W\to {\cal R}_d$, whose general fiber is $\PP^{w-a}$, and whose image is resolved by $\cU\otimes W'$ over $\Gr(a-1,v)\times F'$. By a simple computation, one gets that 
\[
\codim(E)=w-a+1.
\]
Therefore, if the contracted locus $\tilde{E}$ of $\pi$ is the union of such $E$'s, and for each of them $w>a$, no divisor is contained in $\tilde{E}$. This is the case for example of $(F_i,W_i)$ when $r_1,r_2 < p_1 < r_1+r_2$. Similarly, one can work out the case of orbit closures admitting other desingularizations.
\end{remark}

The following are some explicit examples of such loci.

\begin{example}
Take $X=\Gr(4,8)$, $E_1=2(\cO_X(-1)\oplus \cO_X)$, $E_2=\cQ\oplus \cO_X$, $E_3=3\cO_X$, and the orbit closure $Y=\overline{\cO}_{3,1,3}$. Then $D_Y(s)$ is a smooth fourfold with trivial canonical bundle and $\chi(\cO_{D_Y(s)})=2$.
\end{example}

\begin{example}
Take $X=\IGr(2,8)$, $E_1=\cU\oplus \cO_X(-1)$, $E_2=\cU^*\oplus 2\cO_X$, $E_3=2\cO_X$, and the orbit closure $Y=\overline{\cO}_{2,1,2}$. Then $D_Y(s)$ is a fourfold with trivial canonical bundle, singular in codimension three, and whose resolution satisfies $\chi(\cO_{\zero(\tilde{s})})=2$.
\end{example}

\begin{example}
Take $X=\OGr(2,9)$, $E_1=3\cO_X$, $E_2=\cU^*\oplus 2\cO_X$, $E_3=\cU$, and the orbit closure $Y=\overline{\cO}_{2,1,2}$. Then $D_Y(s)$ is a fourfold with trivial canonical bundle singular in codimension three and whose resolution satisfies $\chi(\cO_{\zero(\tilde{s})})=2$.
\end{example}

We next consider the quiver studied in Section \ref{secA3tipo2} (refer to it for the notations).
In the relative setting, we fix a smooth projective variety $X$, three vector bundles $E_1,E_2,E_3$ of ranks $d_1, d_2, d_3$ such that $\Hom(E_2, E_1)\oplus \Hom(E_2,E_3)$ is globally generated. Then, suppose that $s$ is a general section of this bundle, and fix an orbit closure $Y=\overline{\cO}_{k_1, k_2, q_1}$ inside ${\mathcal R}_d=\Hom(\CC^{d_2},\CC^{d_1})\oplus \Hom(\CC^{d_2},\CC^{d_3})$. Recall that $D_Y(s)$ is the locus of points $x\in X$ which are sent by the section $s$ inside $Y\subset (\Hom(E_2,E_1) \oplus \Hom(E_2,E_3))_x \cong {\mathcal R}_d$. 

\begin{theorem}
Let $D_Y(s)$ be defined as above, where $Y=\overline{\cO}_{k_1, k_2, q_1}\subset {\mathcal R}_d$. Then $K_{D_Y(s)}$, $\codim_X(D_Y(s))$ and a lower bound for 
$\codim_{D_Y(s)}\Sing(D_Y(s))$ are given in Table \ref{tableA3ext}. 
\end{theorem}

\begin{proof}
The proof is similar to the one of Theorem \ref{prop.ODLA3}.
\end{proof}

\begin{table}[h!bt]
	\begin{center}
	\begin{footnotesize}
		\caption{ODL from a quiver of type $A_3$ with $a_1(0)=a_2(0)=s_2$. We use the following variables: $\eta_1=d_2-k_1$, $\eta_2=d_2-k_2$.}
		\label{tableA3ext}
		\begin{tabular}{cccc} \toprule
			Case & $K_{D_Y(s)/X}$ & $\codim_X D_Y(s)$ & $\codim_{D_Y(s)}\Sing$
			\\ \midrule
			$(i)$ & $\begin{array}{c} 
			\det E_1^{d_2-k_2}\otimes \\
			 \det E_2^{-d_2+q_1}\otimes \det E_3^{d_2-k_1}
			\end{array}$ & $\begin{array}{c}
			\eta_1^2+\eta_2^2+q_1^2+ \\
			-q_1(\eta_1+\eta_2)
			\end{array}$ & $\geq\min\Bigg\{\begin{array}{c}
			2(\eta_1-q_1)+1\\
			2(\eta_2-q_1)+1\\
			2q_1+1
			\end{array}\Bigg\}$
			\\ \midrule
			$(ii)$ & $\begin{array}{c} 
			\det E_1^{q_1}\otimes \det E_3^{d_1} \otimes \\
			 \det E_2^{-d_1-q_1}
			\end{array}$ & $\begin{array}{c} 
			(d_3-\eta_2)^2-q_1^2\\
			+(d_3-\eta_2)q_1 
			\end{array}$ & $\geq\min\Bigg\{\begin{array}{c}
			2(d_3-\eta_2)+1\\
			2q_1+1
			\end{array}\Bigg\}$		
			\\ \bottomrule
		\end{tabular}
		\end{footnotesize}
	\end{center}
\end{table}

Finally, let $Q$ be the quiver appearing in Section \ref{secA3tipo3} (refer to it for the notations).
In the relative setting, we fix a smooth projective variety $X$, three vector bundles $E_1,E_2,E_3$ of dimension $d_1, d_2, d_3$ such that $\Hom(E_1, E_2)\oplus \Hom(E_2,E_3)$ is globally generated. Then, suppose that $s$ is a general section of this bundle, and fix an orbit closure $Y=\overline{\cO}_{r_1, k_2, u_1}$ inside ${\mathcal R}_d=\Hom(\CC^{d_1},\CC^{d_2})\oplus \Hom(\CC^{d_2},\CC^{d_3})$. Recall that $D_Y(s)$ is the locus of points $x\in X$ which are sent by the section $s$ inside $Y\subset (\Hom(E_1,E_2) \oplus \Hom(E_2,E_3))_x \cong {\mathcal R}_d$.

\begin{theorem}
Let $D_Y(s)$ be defined as above, where $Y=\overline{\cO}_{r_1, k_2, u_1}\subset {\mathcal R}_d$. Then $K_{D_Y(s)}$, $\codim_X(D_Y(s))$ and a lower bound for $\codim_{D_Y(s)}\Sing(D_Y(s))$ are given in Table \ref{tableA3oneway}. 
\end{theorem}

\begin{table}[h!bt]
	\begin{center}
	\begin{footnotesize}
		\caption{ODL from a quiver of type $A_3 $with $a_1(0)=s_1$, $a_2(0)=s_2$. We use the following variables: $\eta_1=k_2-d_1+r_1$, $\eta_2=d_2-u_1$.}
		\label{tableA3oneway}
		\begin{tabular}{cccc} \toprule
			Case & $K_{D_Y(s)/X}$ & $\codim_X D_Y(s)$ & $\codim_{D_Y(s)}\Sing$
			\\ \midrule
			$(i)$ & $\begin{array}{c} 
			\det E_1^{-k_2}\otimes \\
			  \det E_3^{k_2}
			\end{array}$ & $k_2^2+(d_1-r_1)^2$ & $\geq\min\Bigg\{\begin{array}{c}
			2(d_1-r_1)+1\\
			2(d_2-d_1)+1\\
			k_2+1
			\end{array}\Bigg\}$
			\\ \midrule
			$(ii)$ & $\begin{array}{c} 
			\det E_1^{-d_2+r_1}\otimes \\
			 \det E_3^{k_2}
			\end{array}$ & $\begin{array}{c}
			2(d_1-r_1)^2+\eta_1\eta_2+\\
			+(d_1-r_1)(\eta_1+\eta_2)
			\end{array}$ & $\geq 2(d_1-r_1)+1$		
			\\ \bottomrule
		\end{tabular}
		\end{footnotesize}
	\end{center}
\end{table}

\subsubsection{Quiver degeneracy loci of type $D_4$}

For notations, we refer to Section \ref{secD4}.
In order to construct the degeneracy loci, we fix a smooth projective variety $X$, four vector bundles $E_1,E_2,E_3, E_4$ of dimension $d_1, d_2, d_3, d_4$ such that $\Hom(E_1, E_2)\oplus \Hom(E_3,E_2)\oplus \Hom(E_4,E_2)$ is globally generated. Then, suppose that $s$ is a general section of this bundle, and fix an orbit closure $Y=\overline{\cO}_{r}$ inside ${\mathcal R}_d=\Hom(\CC^{d_1},\CC^{d_2})\oplus \Hom(\CC^{d_3},\CC^{d_2})\oplus \Hom(\CC^{d_4},\CC^{d_2})$. Recall that $D_Y(s)$ is the locus of points $x\in X$ which are sent by the section $s$ inside $Y\subset (\Hom(E_1,E_2) \oplus \Hom(E_3,E_2)\oplus \Hom(E_4,E_2))_x \cong {\mathcal R}_d$.

\begin{theorem}
Let $D_Y(s)$ be defined as above, where $Y=\overline{\cO}_{r}\subset {\mathcal R}_d$ is resolved by the resolution of type $(ii)$ or $(iii)$. Then 
\[
K_{D_Y(s)/X}= \det E_1^{r_1-d_2}\otimes \det E_2^{2x}\otimes \det E_3^{r_2-d_2} \otimes \det E_4^{r_3-d_2},
\]
\[
\codim_X(D_Y(s))=\sum_i r_i^2+x^2,
\]
\[
\codim_{D_Y(s)}\Sing (D_Y(s))\geq \min\Bigg\{\begin{array}{c}
			d_1+x-2r_1+1\\
			d_2+r_1+r_2+r_3-2x+1\\
			d_3+x-2r_2+1\\
			d_4+x-2r_3+1
			\end{array}\Bigg\}.
\]
\end{theorem}

\begin{proof}
The proof is similar to the one of Theorem \ref{prop.ODLA3}.
\end{proof}

The easiest way to construct varieties with trivial canonical bundle in this case is to assume that $E_1$, $E_3$, $E_4$ are trivial and $E_2$ is globally generated. We give two explicit examples of ODLs.

\begin{example}
Take $X$ the intersection of a section of $\cO(1)$ and a section of $\cO(2)$ inside $\Gr(3,7)$; moreover take $E_1=2\cO_X$, $E_2=\cU^*$, $E_3=2\cO_X$, and $E_4=2\cO_X$. The orbit closure chosen will be $Y=\overline{\cO}_{r_1,r_2,r_3,x}=\overline{\cO}_{1,1,1,2}$. Then $D_Y(s)$ is a singular (over a finite number of points) threefold whose resolution of singularities is of type Calabi-Yau. 

This singular variety is a hypersurface inside two singular (over a curve) almost Fano fourfolds $F_1$, $F_2$, which are the corresponding degeneracy loci when the base variety $X$ is a section of respectively $\cO(1)$ and $\cO(2)$ inside $\Gr(3,7)$. We computed some invariants of their resolutions $\tilde{F}_1$ and $\tilde{F}_2$:
 \begin{table}[h!bt]
	\begin{center}
	\begin{small}
		\caption{Some invariants of $\tilde{F}_1$ and $\tilde{F}_2$}
		\label{tab:4foldsFano}
		\begin{tabular}{ccccc} \toprule
			 $i$ & $(-K_{\tilde{F}_i})^4$ & $\chi(\Omega^1_{\tilde{F}_i})$ & $\chi(\Omega^2_{\tilde{F}_i})$ & $\chi(-K_{\tilde{F}_i})$ 
			\\ \midrule
			$1$ & $224$ & $-4$ & $8$ & $51$ 
			\\	\midrule
			$2$ & $28$ & $-16$ & $94$ & $12$
			\\ \bottomrule
		\end{tabular}
		\end{small}
	\end{center}
\end{table}

\end{example}

\begin{example}
Take $X$ the intersection of two sections of $\cO(1)$ inside $\Gr(3,7)$; moreover take $E_1=2\cO_X$, $E_2=\cU^*$, $E_3=2\cO_X$, and $E_4=2\cO_X$. The orbit closure chosen will be $Y=\overline{\cO}_{r_1,r_2,r_3,x}=\overline{\cO}_{1,1,1,2}$. Then $D_Y(s)$ is a singular (over a finite number of points) almost Fano threefold $F$ whose resolution $\tilde{F}$ has the following invariants: $(-K_{\tilde{F}})^3=14$, $\chi(\Omega^1_{\tilde{F}})=-2$, $\chi(-K_{\tilde{F}})=10$.
\end{example}

\begin{example}
Take $X$ a section of $\cO(3)$ inside $\Gr(3,7)$; moreover take $E_1=2\cO_X$, $E_2=\cU^*$, $E_3=2\cO_X$, and $E_4=2\cO_X$. The orbit closure chosen will be $Y=\overline{\cO}_{r_1,r_2,r_3,x}=\overline{\cO}_{1,1,1,2}$. Then $D_Y(s)$ is a singular (over a curve) fourfold whose resolution has trivial canonical bundle and $\chi(\cO_{\zero(\tilde{s})})=2$.
\end{example}

\bibliographystyle{alpha}
\bibliography{bibliiotquiver}

\begin{thebibliography}{BFMT17}

\bibitem[BFMT]{BFMTdue}
V.~{Benedetti}, S.~A. {Filippini}, L.~{Manivel}, and F.~{Tanturri}.
\newblock {Orbital degeneracy loci II: Gorenstein orbits}.
\newblock {\em in preparation}.

\bibitem[BFMT17]{BFMT}
V.~{Benedetti}, S.~A. {Filippini}, L.~{Manivel}, and F.~{Tanturri}.
\newblock {Orbital degeneracy loci and applications}.
\newblock {\em ArXiv preprints:1704.01436}, 2017.

\bibitem[BZ01]{Bobinski}
G.~Bobi{\'{n}}ski and G.~Zwara.
\newblock Normality of orbit closures for {D}ynkin quivers of type {$A_n$}.
\newblock {\em Manuscripta Math.}, 105(1):103--109, May 2001.

\bibitem[DR76]{DlabRingel}
V.~Dlab and C.~M. Ringel.
\newblock {\em {Indecomposable representations of graphs and algebras}}.
\newblock American Mathematical Society, 1976.

\bibitem[Gab72]{Gabriel}
P.~Gabriel.
\newblock Unzerlegbare {D}arstellungen. {I}.
\newblock {\em Manuscripta Math.}, 6:71--103; correction, ibid. 6 (1972), 309,
  1972.

\bibitem[GS]{Macaulay}
Daniel~R. Grayson and Michael~E. Stillman.
\newblock Macaulay2, a software system for research in algebraic geometry.
\newblock Available at \url{http://www.math.uiuc.edu/Macaulay2/}.

\bibitem[Rei03]{Reineke}
M.~Reineke.
\newblock Quivers, desingularizations and canonical bases.
\newblock In {\em Studies in memory of {I}ssai {S}chur ({C}hevaleret/{R}ehovot,
  2000)}, volume 210 of {\em Progr. Math.}, pages 325--344. Birkh\"auser
  Boston, Boston, MA, 2003.

\bibitem[Sut13]{Sutar}
K.~Sutar.
\newblock Resolutions of defining ideals of orbit closures for quivers of type
  {$A_3$}.
\newblock {\em J. Commut. Algebra}, 5(3):441--475, 2013.

\bibitem[Wey03]{Weyman}
J.~Weyman.
\newblock {\em Cohomology of Vector Bundles and Syzygies}.
\newblock Cambridge Tracts in Mathematics. Cambridge University Press, 2003.

\end{thebibliography}


\begin{thebibliography}{KKRS05}

	
	\bibitem[Ber09]{Ber09}
	Marie-Am\'elie Bertin.
	\newblock Examples of {C}alabi-{Y}au 3-folds of {$\mathbb{P}^7$} with
	{$\rho=1$}.
	\newblock {\em Canad. J. Math.}, 61(5):1050--1072, 2009.

	
\end{thebibliography}

\makeatletter
\providecommand\@dotsep{5}
\makeatother
\listoftodos\relax

\end{document}